\documentclass[12pt]{amsart} 
\newcommand{\filename}{qual-quant-subspace-scattering-22-June-2023.tex} 
\usepackage{amsfonts, enumitem,amsmath,amsthm,amssymb,mathrsfs}
\usepackage{tikz-cd,xcolor}
\renewcommand{\geq}{\geqslant}
\renewcommand{\leq}{\leqslant}
\newcommand{\Osh}{{\mathcal O}}                        

\renewcommand{\H}{\mathrm{H}}                          

\newcommand{\Ish}{\mathcal{I}}

\newcommand{\N}{\operatorname{N}}

\renewcommand{\emptyset}{\varnothing}
\newcommand{\KK}{\mathbf{K}}
\newcommand{\FF}{\mathbf{F}}
\newcommand{\NN}{\mathbb{N}} 
\newcommand{\PP}{\mathbb{P}} 
\newcommand{\QQ}{\mathbb{Q}} 
\newcommand{\RR}{\mathbb{R}} 

\newtheorem{theorem}{Theorem}[section]
\newtheorem{lemma}[theorem]{Lemma}
\newtheorem{corollary}[theorem]{Corollary}
\newtheorem{proposition}[theorem]{Proposition}

\theoremstyle{definition}
\newtheorem{defn}[theorem]{Definition}

\newtheorem{example}[theorem]{Example}

\newtheorem{setting}[theorem]{Setting}

\numberwithin{equation}{section}

\begin{document}

\title[Qualitative aspects of the subspace theorem]{On qualitative aspects of the quantitative subspace theorem}

\author{Nathan Grieve}
\address{
School of Mathematics and Statistics, 4302 Herzberg Laboratories, Carleton University, 1125 Colonel By Drive, Ottawa, ON, K1S 5B6, Canada \\
D\'{e}partement de math\'{e}matiques, Universit\'{e} du Qu\'{e}bec \`a Montr\'{e}al, Local PK-5151, 201 Avenue du Pr\'{e}sident-Kennedy, Montr\'{e}al, QC, H2X 3Y7, Canada  \\
Department of Pure Mathematics, University of Waterloo, 200 University Avenue West, Waterloo, ON, N2L 3G1, Canada
}

\email{nathan.m.grieve@gmail.com}%

\begin{abstract} 
We deduce Diophantine arithmetic inequalities for big linear systems and with respect to finite extensions of number fields.  Our starting point is the Parametric Subspace Theorem, for linear forms, as formulated by Evertse and Ferretti  
\cite{Evertse:Ferretti:2013}.  Among other features, this viewpoint allows for a partitioning of the linear scattering, for the Diophantine Exceptional set, that arises in the Subspace Theorem.  Our perspective builds on our work \cite{Grieve:points:bounded:degree}, combined with earlier work of Evertse and Ferretti, \cite{Evertse:Ferretti:2013}, 
 Evertse and Schlickewei, \cite{Evertse:Schlickewei:2002}, and others.  
As an application, we establish a novel linear scattering type result for the Diophantine exceptional set that arises in the main Diophantine arithmetic inequalities of Ru and Vojta \cite{Ru:Vojta:2016}.  This result expands, refines and complements our earlier works (including \cite{Grieve:2018:autissier} and \cite{Grieve:points:bounded:degree}).  A key tool to our approach is the concept of \emph{linear section} with respect to a linear system.  This was defined in \cite{Grieve:points:bounded:degree}.  Another point, which we develop in this article, is a notion of logarithmic \emph{twisted height functions} for local Weil functions and linear systems.    As an additional observation, which is also of an independent interest, we use the theory of Iitaka fibrations to determine the asymptotic nature of such linear sections.  
\end{abstract}
\thanks{
\emph{Mathematics Subject Classification (2020): 11J87, 14G05, 11G50}. \\
\emph{Key Words: Parametric Subspace Theorem, Weil and height functions, Diophantine approximation, Geometry of Numbers, Linear Series.}  \\
I hold grants RGPIN-2021-03821 and DGECR-2021-00218 from the Natural Sciences and Engineering Research Council of Canada. \\
Date: \today.  File: \filename.  \\
}

\maketitle

\section{Introduction}

Our starting point here is the refinement of the Quantitative Subspace Theorem, which was given by Evertse and Ferretti 
\cite{Evertse:Ferretti:2013}.  It improved on earlier work of Evertse and Schlickewei, \cite{Evertse:Schlickewei:2002}, and Evertse, \cite{Evertse:1996}, and was derived as a consequence of the Absolute Parametric Subspace Theorem (\cite[Theorems 2.1, 2.2 and 2.3]{Evertse:Ferretti:2013}).  

In this article, our purpose is to improve upon qualitative aspects of these works.  We are motivated by the recent progress in our understanding of Diophantine approximation and K-stability for projective varieties.  (See for example \cite{Autissier:2011}, \cite{McKinnon-Roth}, \cite{Ru:Vojta:2016}, \cite{Grieve:Function:Fields}, \cite{Ru:Wang:2016}, \cite{Grieve:2018:autissier}, \cite{Heier:Levin:2017},  \cite{Levin:GCD}, \cite{Grieve:toric:gcd:2019}, \cite{Grieve:Divisorial:Instab:Vojta}, \cite{Grieve:points:bounded:degree},  \cite{Grieve:MVT:2019}, \cite{Grieve:chow:approx}, \cite{Grieve:HN:polygons:Laws:Large:Numbers}, \cite{He:Ru:2022} and the references therein.)  
  
Building on the viewpoint of Schmidt \cite{Schmidt:1993} and Evertse \cite{Evertse:1996}, 
here, we formulate a concept of \emph{density} for rational points with respect to a linear system.  Briefly, a collection of rational points is \emph{dense}, with respect to a given linear system, if it is contained in no proper finite union of the linear system's linear sections.  We refer to Section \ref{linear:systems:rational:maps:v:adic:distances}, see also \cite[Definition 3.1]{Grieve:points:bounded:degree}, for details in regards to the notion of linear sections and density of rational points with respect to a given linear system.  In Section \ref{Weil:height:prelims}, we give a construction of local Weil functions, with respect to a given extension of number fields, via presentations of Cartier divisors.  This builds on \cite[Section 2]{Grieve:2018:autissier}, \cite[Section 3]{Grieve:Divisorial:Instab:Vojta} and \cite[Chapter 2]{Bombieri:Gubler}.  

Our first result is a novel logarithmic form of the Parametric Subspace Theorem  (see Theorem \ref{logarithmic:parametric:subspace:thm}).  It gives inequalities that involve \emph{twisted logarithmic height functions} for big line bundles on projective varieties.    The main context that we consider is Setting \ref{Set:up} below.  
It resembles that of \cite[Theorem 2.10]{Ru:Vojta:2016}, \cite[Proposition 4.2]{Autissier:2011} and \cite[Proposition 2.1]{Grieve:2018:autissier}.

\begin{setting}\label{Set:up} 
Let $\KK$ be a number field, $M_{\KK}$ its set of places and $S \subset M_{\KK}$ a finite set.  Let $\overline{\KK}$ be an algebraic closure of $\KK$ and $\FF / \KK$ a finite extension field $\KK \subseteq \FF \subseteq \overline{\KK}$.

Let $L$ be a big line bundle on a geometrically irreducible projective variety $X$.  Assume that both $X$ and $L$ are defined over $\KK$.  
Let
$$
 V := \H^0(X,L)
$$ 
and set $n := \dim V - 1$.  Unless stated otherwise, we always assume that $n \geq 1$.  

Respectively, denote by $X_{\FF}$ and $L_{\FF}$ the base change of $X$ and $L$ with respect to the field extension $\FF / \KK$.
Put 
$$V_{\FF} := V \otimes_{\KK} \FF = \H^0(X_{\FF}, L_{\FF})\text{.}$$  

If $v \in M_{\KK}$ and if $D$ is a Cartier divisor on $X$ and defined over $\FF$, then $\lambda_{\mathcal{D}}(\cdot,v)$ denotes a local Weil function with respect to the place $v$ and with respect to a fixed choice of presentation $\mathcal{D}$ defined over $\FF$.  For the particular case that $s \in V_{\FF}$ and $D = \operatorname{div}(s)$ we often write $\lambda_s(\cdot,v)$ in place of $\lambda_{\mathcal{D}}(\cdot,v)$. (We refer to Section \ref{Weil:height:prelims} for more details.)
\end{setting}

We deduce Theorems  \ref{logarithmic:FW:subspace:thm} and \ref{logarithmic:linear:scattering:subspace:thm} from the following  
novel logarithmic formulation of the Parametric Subspace Theorem for big linear systems and with respect to a finite extension of number fields.  This is the content of Theorem \ref{logarithmic:parametric:subspace:thm}.  We prove it in Section \ref{parametric:proof}.  It is a consequence of a more robust relative formulation of \cite[p. 515]{Evertse:Ferretti:2013}.  (See Theorem \ref{Parametric:Subspace:Thm}.)

\begin{theorem}[Parametric Subspace Theorem for big linear systems]\label{logarithmic:parametric:subspace:thm}
Consider the situation of Setting \ref{Set:up} and for each $v \in S$ choose a collection of linearly independent global sections 
$$s_{v0},\dots,s_{vn} \in V_{\FF}\text{.}$$  
Suppose that $\epsilon > 0$ is a fixed positive real number.  
Fix a collection of real numbers $c_{vi} \in \RR$ 
which has the property that
$$
\sum_{i=0}^n c_{vi} = 0 
\text{ 
for all 
$v \in S\text{.}$
}
$$
Then there exist a real number 
$Q_0 > 1$ 
and a finite collection of proper linear sections 
$$
\Lambda_1,\dots,\Lambda_{t} \subsetneq X \text{, }
$$
with respect to the linear series $|V|$, 
which are defined over $\KK$ and which have the property that for all 
$Q \geq Q_0$
there is a linear section
$$
\Lambda_{j_Q} \in \left\{ \Lambda_1,\dots,\Lambda_{t} \right\}
$$
which contains all $\KK$-rational points
$$
x \in \left( X \setminus \left( \operatorname{Bs}(|V|) \bigcup  \bigcup_{\substack{v \in S \\ i = 0,\dots,n} }\operatorname{Supp}(s_{vi}) \right) \right)(\KK)
$$
that satisfy the inequalities that
$$
\sum_{v \in S} \left( \lambda_{s_{vi}}(x,v) + c_{vi} \cdot \log(Q) \right) \geq h_L(x) + \epsilon \cdot \operatorname{log}(Q) + \mathrm{O}(1) 
$$  
for all $i = 0,\dots, n$.

In particular, the collection of such points is not dense with respect to $|V|$.
\end{theorem}

As an application of Theorem \ref{logarithmic:parametric:subspace:thm}, we deduce a complementary form of  a  celebrated theorem of Faltings and W\"{u}stholz \cite[Theorem 8.1]{Faltings:Wustholz}.

\begin{theorem}[Faltings and W\"{u}stholz inequalities for big linear systems]\label{logarithmic:FW:subspace:thm}
Consider the situation of Setting \ref{Set:up} and for each $v \in S$
choose a collection of linearly independent global sections 
$$s_{v0},\dots,s_{vn} \in V_{\FF}\text{.}$$  
Finally, fix a collection of real numbers $d_{vi}$, for all $v \in S$ and all $i = 0,\dots,n$, which have the property that
$$
\sum_{v \in S} \sum_{i=0}^n d_{vi} > n+1 \text{.}
$$

Then the set of solutions
$$
x \in \left( X \setminus \left( \operatorname{Bs}(|V|) \bigcup \bigcup_{\substack{v \in S \\ i = 0,\dots,n} }\operatorname{Supp}(s_{vi}) \right) \right)(\KK)
$$
of the system of inequalities
$$
\lambda_{s_{vi}}(x,v) - d_{vi} \cdot h_L(x) + \mathrm{O}_v(1) \geq 0 
$$
for all $i = 0,\dots,n$ and all $v \in S$ is not dense with respect to $|V|$.
\end{theorem}

As an application of Theorem \ref{logarithmic:FW:subspace:thm}, we obtain a \emph{qualitative linear scattering} type result which improves upon our understanding of the Subspace Theorem.  It provides a qualitative perspective to the quantitative work of Evertse \cite{Evertse:1996} and Evertse and Schlickewei \cite{Evertse:Schlickewei:2002}.
We refer to  \cite[p. 240]{Evertse:1996} for the concept of \emph{linear scattering} in the classical case of the Subspace Theorem.

The proof of Theorem \ref{logarithmic:linear:scattering:subspace:thm}, see Section \ref{Log:Subspace:Proof}, expands on the approach of \cite[Section 21]{Evertse:Schlickewei:2002}.  It follows the suggestion given in \cite[p.~ 514]{Evertse:Ferretti:2013}.  A key role is played by Lemma \ref{Evertse:scattering:lemma4} which is a logarithmic form of \cite[Lemma 4]{Evertse:1984}.

\begin{theorem}[Linear scattering and the Subspace Theorem for big linear systems]\label{logarithmic:linear:scattering:subspace:thm}
Consider the situation of Setting \ref{Set:up} and 
for each $v \in S$ fix a collection of linearly independent global sections 
$$s_{v0},\dots,s_{vn} \in V_{\FF}\text{.}$$  
Fix a positive and sufficiently small real number  
$\epsilon > 0$ and consider the set of solutions
$$
x \in \left(X \setminus \left( \operatorname{Bs}(|V|) \bigcup   \bigcup_{\substack{v \in S \\ i = 0,\dots,n}} \operatorname{Supp}(s_{vi}) \right) \right)(\KK) 
$$
with sufficiently large height 
$$h_L(x) \gg 0$$ 
to the inequality
\begin{equation}\label{Schmidt:height:inequalities}
\sum_{v \in S} \sum_{i=0}^n \lambda_{s_{vi}}(x,v) \geq (n+1+\epsilon)h_L(x) + \mathrm{O}(1) \text{.}
\end{equation}
Then this solution set is not dense with respect to $|V|$.  

In more specific terms, the solution set admits a decomposition into finitely many subsets such that for each subset there exists a collection of real numbers $e_{vi}$, for each $v \in S$ and all $i = 0,\dots,n$, such that
$$
\sum_{v \in S} \sum_{i=0}^n e_{vi} > n+1
$$
and such that all solutions in this subset satisfy the inequalities that
$$
\lambda_{s_{vi}}(x,v) - e_{vi} \cdot h_L(x) + \mathrm{O}_v(1) \geq 0 
$$
for all $v \in S$ and all $i = 0,\dots,n$.

In particular, the collection of solutions to \eqref{Schmidt:height:inequalities} that are in the subset corresponding to the weights $e_{vi}$ is contained in a finite union of proper linear sections of $X$ with respect to the linear series $|V|$.
\end{theorem}

Theorem \ref{logarithmic:linear:scattering:subspace:thm} implies a more general result, which is in the spirit of Vojta's formulation of the Subspace Theorem for hyperplanes in general position.   (See for instance \cite[Theorem 7.2.9]{Bombieri:Gubler}.) We formulate this form of the Subspace Theorem, for big linear systems and global sections in \emph{linearly general position}, as Corollary \ref{general:position:logarithmic:linear:scattering:subspace:thm} below.

\begin{corollary}[Subspace Theorem for big linear systems and global sections in linearly general position]\label{general:position:logarithmic:linear:scattering:subspace:thm}
Consider the situation of Setting \ref{Set:up} and for all $v \in S$ let 
$$s_{v0},\dots,s_{vn_v} \in V_{\FF}$$ 
be a collection of global sections of $L_{\FF}$, with $n_v \geq n$, which have the property that all subsets of cardinality not exceeding $n+1$ are $\FF$-linearly independent.  
Fix a positive and sufficiently small real number $\epsilon >0$ and consider the set of solutions
$$
x \in \left(X \setminus \left( \operatorname{Bs}(|V|) \bigcup   \bigcup_{\substack{v \in S \\ i = 0,\dots,n}} \operatorname{Supp}(s_{vi}) \right) \right)(\KK) 
$$
with sufficiently large height 
$$h_L(x) \gg 0$$ 
to the inequality
\begin{equation}\label{Schmidt:height:inequalities:gen:pos}
\sum_{v \in S} \sum_{i=0}^{n_v} \lambda_{s_{vi}}(x,v) \geq (n+1+\epsilon)h_L(x) + \mathrm{O}(1) \text{.}
\end{equation}
Then this solution set is not dense with respect to $|V|$.  

In more specific terms, the solution set is a finite union of proper linear sections with respect to $|L|$.  Furthermore, it admits a decomposition into finitely many subsets such that for each subset there exist real numbers $e_{vi}$, for each $v \in S$ and all $i = 0,\dots,n$, such that
$$
\sum_{v \in S} \sum_{i=0}^n e_{vi} > n+1 
$$
and such that for some linearly independent collection of sections 
$$\{s_{vj_0},\dots,s_{vj_n}\} \subseteq \{s_{v0},\dots,s_{vn_v}\} $$ 
all except for perhaps finitely many solutions in this subset satisfy the inequalities that
$$
\lambda_{s_{vj_i}}(x,v) - e_{vi} \cdot h_L(x) + \mathrm{O}_v(1) \geq 0 
$$
for all $v \in S$ and all $i = 0,\dots,n$.

In particular, the collection of solutions to \eqref{Schmidt:height:inequalities:gen:pos} that are in the subset corresponding to the weights $e_{vi}$ are contained in a finite union of proper linear sections of $X$ with respect to the linear series $|V|$.
\end{corollary}

As an application of Theorem \ref{logarithmic:linear:scattering:subspace:thm}, in the form of Corollary \ref{general:position:logarithmic:linear:scattering:subspace:thm}, we deduce qualitative scattering information that arises in the conclusion of the Ru-Vojta Arithmetic General Theorem (\cite[p. 964]{Ru:Vojta:2016}).  This is the content of Theorem \ref{Arithmetic:General:Thm:linear:scattering}.  In its conclusion, the description of the Diophantine exceptional set refines, for the case of $\KK$-rational points, and builds on our main result from \cite{Grieve:points:bounded:degree}.  The statement of Theorem \ref{base:loci} requires the notion of stable base locus.  We refer to Section \ref{base:loci} for further details.

\begin{theorem}[Arithmetic General Theorem with linear scattering]\label{Arithmetic:General:Thm:linear:scattering}
Working over a base number field $\KK$, fix a finite set of places $S \subset M_{\KK}$.  Let $L$ be a big line bundle on a geometrically irreducible projective variety $X$.  Assume that $X$ and $L$ are both defined over $\KK$.  Let $D_1,\dots,D_q$ be a collection of nonzero properly intersecting effective Cartier divisors on $X$ and defined over $\FF$.  Then there are optimal constants $\eta(L,D_i)$, for $i=1,\dots,q$, which are such that if $\epsilon > 0$  is a sufficiently small real number, then the inequality
$$
\sum_{i=1}^q \eta(L,D_i)m_S(x,D_i) \leq (1+\epsilon)h_L(x)
$$
holds true for all $x \in X(\KK)$ outside of a Zariski closed subset $Z \subsetneq X$.  Here, $m_S(\cdot,D_i)$, for $i = 1,\dots,q$, is the proximity function of $D_i$ with respect to $S$. 

Moreover, the \emph{Diophantine exceptional} set $Z$ may be described as
$$
Z = \operatorname{Bs}(L) \bigcup \left( \bigcup_{i=1}^q \operatorname{Supp}(D_i) \right) \bigcup \left( \Lambda_1 \bigcup \hdots \bigcup \Lambda_{\ell} \right) 
$$
for $\operatorname{Bs}(L)$ the stable base locus of $L$ and $\Lambda_1,\dots,\Lambda_{\ell}$ linear sections of the complete linear system $|L^{\otimes m}|$ for some suitably large positive integer $m$.  

Finally, all but perhaps finitely many points
in the union of the linear sections $\Lambda_1\bigcup \hdots \bigcup \Lambda_{\ell}$, admit a \emph{linear scattering type decomposition} into finitely many subsets such that the following is true: for each subset there exist real numbers $e_{vi}$, for all $v \in S$ and all $i = 0,\dots,n_m$, such that
$$
\sum_{v \in S} \sum_{i=0}^{n_m} e_{vi} > n_m + 1
$$ 
and for each  $v \in S$, linearly independent sections 
$$s_{vj_0},\dots, s_{vj_{n_m}} \in \H^0(X_{\FF},L^{\otimes m}_{\FF})$$ 
which are 
such that all solutions in this subset satisfy the inequalities that
$$
\lambda_{s_{vj_i}}(x,v) - e_{vi} h_{L^{\otimes m}}(x) + \mathrm{O}_v(1) \geq 0 
$$
for all $v \in S$ and all $i = 0,\dots,n_m$.
\end{theorem}

Theorem \ref{Arithmetic:General:Thm:linear:scattering} complements our results from \cite{Grieve:points:bounded:degree}.  Its proof is given in Section \ref{Proof:Arithmetic:General:Thm:linear:scattering} and is based on the Ru-Vojta filtration construction (which has origins in the work of Corvaja-Zannier \cite{Corvaja:Zannier:2002}, Levin \cite{Levin:2009}, Autissier \cite{Autissier:2011}, Ru \cite{Ru:2017} and others).  In \cite{Grieve:points:bounded:degree}, this is exposed in detail and expanded upon to treat the case of points of bounded degree.   
Here, our novel description of the Diophantine exceptional set that arises in its conclusion is made possible by applying the Subspace Theorem, in the form of Theorem \ref{logarithmic:linear:scattering:subspace:thm}, which we derive here from its parametric formulation (Theorem \ref{logarithmic:parametric:subspace:thm}).  

It is also important to make note of the bigness assumption in the statement of Theorems \ref{logarithmic:FW:subspace:thm}, \ref{logarithmic:linear:scattering:subspace:thm} and \ref{Arithmetic:General:Thm:linear:scattering} and Corollary \ref{general:position:logarithmic:linear:scattering:subspace:thm}.  Indeed, as is indicated in the proof of these results, the Northcott property for big line bundles plays an important role.  
On the other hand, recall that there exists Subspace Theorem inequalities for linear systems, without  the assumption of bigness for  the given linear system.   (See for instance \cite[Theorem 2.10]{Ru:Vojta:2016} and \cite[Theorem 3.3]{Grieve:points:bounded:degree}.)  Here the bigness assumption is used to deduce the linear scattering Subspace Theorem result (Theorem \ref{logarithmic:linear:scattering:subspace:thm}) from Theorem \ref{logarithmic:parametric:subspace:thm}.

As some final observations, and to help place matters into perspective, in Section \ref{Iitaka:fibration:linear:sections} we apply the theory of Iitaka fibrations to determine the \emph{asymptotic nature} of the \emph{linear sections}  that are associated to a linear series.  We refer to Section \ref{linear:systems:rational:maps:v:adic:distances}, see Definition \ref{linear:section:defn}, for a precise definition of our concept of linear section with respect to a linear system.  It builds on \cite[Definition 3.1]{Grieve:points:bounded:degree}.

\subsection*{Acknowledgements} 
I thank the Natural Sciences and Engineering Research Council of Canada for their support via my grants RGPIN-2021-03821 and DGECR-2021-00218. 
This work benefited from trips to BIRS, Banff, during the Summer of 2022.  It is my pleasure to thank colleagues for their interest, encouragement and engagement on related topics.   Finally, I thank anonymous referees for carefully reading my manuscript and for offering helpful suggestions.

\section{Preliminaries}\label{Prelims}

In this article, our conventions and notations closely resemble those of \cite{Bombieri:Gubler} and \cite{Laz}.  We briefly indicate some of the main points here.

\subsection{Number fields, absolute values and multiplicative projective heights}

Let $\KK$ be a number field with set of places $M_{\KK}$ and fix an algebraic closure $\overline{\KK}$.  
  
If 
$v \in M_{\KK}\text{,}$ 
then $|\cdot|_v$ is its normalized absolute value.  Thus, as in \cite[p. 11]{Bombieri:Gubler}, if 
$v \in M_{\KK}$ 
lies above $p \in M_{\QQ}\text{,}$ 
then the restriction of $|\cdot|_v$ to $\QQ$ is $|\cdot|_p^{[\KK_v:\QQ_p]/[\KK:\QQ]}$.  Here, $\KK_v$ and $\QQ_p$ are the respective completions of $\KK$ and $\QQ$ at $v$ and $p$.  By these conventions, the product formula holds true with multiplicities equal to one.  Explicitly
$$
\prod_{v \in M_{\KK}} | \alpha |_v = 1 \text{ 
for $\alpha \in \KK^\times$.}
$$

If $\FF / \KK$ is a finite extension field, $\KK \subseteq \FF \subseteq \overline{\KK}$, and $v \in M_{\KK}$, then choose $w \in M_{\FF}$ with $w \mid v$ and put
$$
|\cdot|_{v,\KK} = |\cdot|_{v,\FF / \KK} := \left| \N_{\FF_w / \KK_v} (\cdot) \right|_{v}^{\frac{1}{[\FF_w:\KK_v]}} \text{.}
$$
Then $|\cdot|_{v,\KK}$ is an extension of $|\cdot|_v$ to $\FF$.

In terms of multiplicative projective heights, recall, that if 
$$
\mathbf{x} = [x_0:\dots:x_n] \in \PP^n(\KK)
$$
then its \emph{multiplicative height} with respect to the tautological line bundle $\Osh_{\PP^n}(1)$ is defined to be
$$
H_{\Osh_{\PP^n}(1)}(\mathbf{x}) = \prod_{v \in M_{\KK}} |\mathbf{x}|_v  \text{ 
where }
|\mathbf{x}|_v = \max_{i=0,\dots,n}|x_i|_v \text{.}
$$

\subsection{Local Weil and logarithmic height functions}\label{Weil:height:prelims}
Consider the situation of Setting \ref{Set:up}.  Building on the approach of  \cite[Section 2]{Grieve:2018:autissier} and \cite[Section 3]{Grieve:Divisorial:Instab:Vojta} we develop further the theory of local Weil functions with respect to the field extension $\FF / \KK$.  

Let $D$ be a Cartier divisor on $X_{\FF}$ with line bundle $\Osh_{X_{\FF}}(D)$ and meromorphic section $s = s_D$.  By a slight abuse of notation we also say that $D$ is a Cartier divisor on $X$ and defined over $\FF$.  Further we understand the set $(X \setminus \operatorname{Supp}(D))(\KK)$ to mean the set of those $\KK$-rational points $x \in X(\KK)$ whose image in $X_{\FF}(\FF)$ do not lie in the support of $D$. When no confusion is likely, in what follows, we employ variants of this notation.

Fixing globally generated line bundles $N$ and $M$ on $X_{\FF}$, with the property that
$$
\Osh_{X_{\FF}}(D) \simeq N \otimes M^{-1}
$$
together with a choice of respective generating sections 
$$\mathbf{s} := (s_0,\dots,s_k) \text{ and } \mathbf{t} := (t_0,\dots,t_\ell)$$ 
yields the data of a \emph{presentation} for $D$ (and defined over $\FF$).  (Compare with \cite[\S 2.2.1]{Bombieri:Gubler}.)

For example, as in \cite[Exercise II.7.5]{Hart}, if $N$ is very ample and if for some positive integer $m >0$ the line bundle
$$N^{\otimes m} \otimes \Osh_{X_{\FF}}(D)$$
is globally generated, then setting
$$M := N^{\otimes (m +1)}$$ 
we may write
$$\Osh_{X_{\FF}}(D) \simeq M^{-1} \otimes \left(N^{\otimes (m + 1)} \otimes \Osh_{X_{\FF}}(D)\right)\text{;}$$ 
in doing so, we obtain an expression of $\Osh_{X_{\FF}}(D)$ as a difference of two very ample line bundles.

Denoting the data of such a presentation as
\begin{equation}\label{Weil:function:eqn3}
\mathcal{D} = (s_D; N,\mathbf{s}; M, \mathbf{t})
\end{equation}
the corresponding \emph{local Weil function for $D$} with respect to a place 
 $v \in M_{\KK}$ 
 is given by
\begin{equation}\label{Weil:function:eqn11}
\lambda_{\mathcal{D}}(x,v) = \lambda_s(x,v) := \max_{j=0,\dots,k} \min_{i=0,\dots,\ell} \left| \frac{s_j}{t_i s_D }(x) \right|_{v, \KK} 
\end{equation}
for
\begin{equation}\label{Weil:function:eqn11:prime}
x \in \left( X \setminus \operatorname{Supp}(D) \right)(\KK) \text{.}
\end{equation}
In \eqref{Weil:function:eqn11}, we have fixed $w \in M_{\FF}$ with $w \mid v$.  

Note that in \eqref{Weil:function:eqn11:prime} if $D$ is strictly defined over $\FF$, in the sense that $D$ does not descend with respect to the field extension $\FF / \KK$, then
$$
\left( X \setminus \operatorname{Supp}(D) \right)(\KK) = X(\KK) \text{.}
$$
In either case, the intuitive sense for \eqref{Weil:function:eqn11} is to get a measure of the $v$-adic size of, or rather of the negative logarithmic $v$-adic distance to, the meromorphic function $s_D$, which is defined over $\FF$ and not necessarily over $\KK$, when evaluated at $X$'s $\KK$-points.

Recall, the finite set of places $S \subset M_{\KK}$.  The \emph{proximity function} of $D$ with respect to $S$ is defined to be
$$
m_S(x,D) := \sum_{v \in S} \lambda_{\mathcal{D}}(x,v) \text{.}
$$ 

As in  \cite[Theorem 2.2.11]{Bombieri:Gubler}, if $\mathcal{D}'$ and $\mathcal{D}$ are two presentations of $D$, then the corresponding local Weil functions with respect to a place 
 $v \in M_{\KK}$ 
 are related by
$$
\lambda_{\mathcal{D}'}(\cdot,v) = \lambda_{\mathcal{D}}(\cdot,v) + \mathrm{O}(1) \text{.}
$$

Further, as in \cite[Proposition 2.3.9]{Bombieri:Gubler},
 if $D$ is an effective Cartier divisor on $X_{\FF}$, then there exists a presentation $\mathcal{D}$ which has the property that
$$
\lambda_{\mathcal{D}}(x,v) \geq 0 \text{ for all }
x \in \left( X \setminus \operatorname{Supp}(D)\right)(\KK) \text{.}
$$

The concept of presentation for Cartier divisors is significant from the viewpoint of locally bounded metrics and local Weil functions.  (See \cite[\S 2.2--2.7]{Bombieri:Gubler} for more details.)   This is illustrated by the following example, which is important to what we do here.

\begin{example}[Compare with {\cite[Examples 2.7.4 and 2.7.7]{Bombieri:Gubler}}]\label{standard:metric:taut:PPn}
Fix a place $v \in M_{\KK}$.  On projective $n$-space $\PP^n$, the tautological line bundle $\Osh_{\PP^n}(1)$ has the standard metric $||\cdot||_{v,\KK}$.  It is \emph{locally bounded}, in the sense of \cite[Definition 2.7.1]{Bombieri:Gubler}, and is defined by the condition that
\begin{equation}\label{Weil:function:eqn4}
||\ell(\mathbf{x})||_{v,\KK} = ||\ell(\mathbf{x})||_{v,\FF / \KK} := \frac{|\ell(\mathbf{x})|_{v,\KK} }{ \max\limits_{0 \leq j \leq n } |x_j|_{v,\KK}} 
\end{equation}
for each linear form
\begin{equation}\label{Weil:function:eqn5}
\ell(x) \in \FF[x_0,\dots,x_n] \text{.}
\end{equation}

Each such linear form \eqref{Weil:function:eqn5} determines a presentation
\begin{equation}\label{Weil:function:eqn6}
\mathcal{H} := ( \ell(x); \Osh_{\PP^n_{\FF}}(1), (x_0,\dots,x_n); \Osh_{\PP^n_{\FF}}, (1) )
\end{equation}
of the hyperplane
$$
H := \operatorname{div}(\ell(x))
$$
that it defines.  There is a \emph{local Weil function}
$$
\lambda_{\mathcal{H}}(\mathbf{x},v) = \lambda_{\ell(x)}(\mathbf{x},v)
$$
for $H$ with respect to the presentation \eqref{Weil:function:eqn6} and the place $v$.  It has domain the set of points
$$
\mathbf{x} \in \left( \PP^n \setminus \operatorname{Supp}(H) \right)(\KK)
$$
and is defined by the condition that
\begin{align}\label{Weil:function:eqn10}
\begin{split}
\lambda_{\mathcal{H}}(\mathbf{x},v) & = - \log ||\ell(\mathbf{x})||_{v,\KK} \\
& = \max_{j = 0,\dots,n} \log \left| \frac{x_j}{\ell(\mathbf{x})} \right|_{v, \KK} \text{.}
\end{split}
\end{align}

The local Weil function \eqref{Weil:function:eqn10} is the local Weil function that is determined by $||\cdot||_{v,\KK}$, the locally bounded metric \eqref{Weil:function:eqn4} on the tautological line bundle $\Osh_{\PP^n}(1)$, with respect to the presentation \eqref{Weil:function:eqn6}.
\end{example}

Returning to general considerations, recall the construction of logarithmic height functions from the viewpoint of local Weil functions \cite[\S 2.3.3]{Bombieri:Gubler}.  Let $D$ be a Cartier divisor on $X$ defined over $\KK$ and with fixed presentation
$$
\mathcal{D} := (s_D; N, \mathbf{s}; M, \mathbf{t})
$$
defined over $\KK$.
Then for each $\KK$-point
$
x \in X(\KK)
$
there exist global sections 
$$s_i \in \H^0(X,N) \text{ and } t_j  \in \H^0(X,M) $$ 
which have the property that
$$
s_i(x) \not = 0 \text{ and } t_j(x) \not = 0 \text{.}
$$

As a consequence, the line bundle $\Osh_X(D)$ admits a meromorphic section 
$s := s_i \otimes t_j^{-1}\text{,}$ 
so that if 
$
D(s) := \operatorname{div}(s) 
$
is the Cartier divisor that corresponds to $s$, then
$
x \not \in \operatorname{Supp}(D(s)) 
$.

In particular, 
$$
\mathcal{D}(s) := (s; N, \mathbf{s}; M, \mathbf{t})
$$
is a presentation of the Cartier divisor $D(s)$.  

In this way, up to a constant term $\mathrm{O}(1)$, the \emph{logarithmic height function} $h_{\Osh_X(D)}(\cdot)$ may be described as
$$
h_{\Osh_X(D)}(x) := \sum_{v \in M_{\KK}}\lambda_{\mathcal{D}(s)}(x,v) + \mathrm{O}(1) \text{.}
$$

We conclude this subsection by recalling the description of logarithmic projective heights via the viewpoint of presentations of Cartier divisors.

\begin{example}[{\cite[Example 2.3.2]{Bombieri:Gubler}}]\label{log:height}
The coordinate hyperplane 
$$
H := \{ \mathbf{x} \in \PP^n(\KK) : x_0 = 0 \} \subseteq \PP^n_{\KK}
$$
admits the presentation
$$
\mathcal{H} = (x_0 ; \Osh_{\PP^n}(1), (x_0,\dots,x_n); \Osh_{\PP^n}, (1)) \text{.}
$$
Thus, if $\mathbf{x} \in \PP^n(\KK)$
and 
$x_0 \not = 0\text{,}$
then
\begin{align*}
\begin{split}
h_{\Osh_{\PP^n}(1)}(\mathbf{x}) & = \sum_{v \in M_{\KK}} \max_k \log \left| \frac{x_k}{x_0} \right|_v \\
& = \sum_{v \in M_{\KK}} \lambda_{\mathcal{H}}(\mathbf{x},v) \text{.}
\end{split}
\end{align*}
\end{example}

\subsection{Asymptotics of linear series}\label{base:loci}
We fix some notation and conventions and recall a few concepts that pertain to asymptotic aspects of linear series.  Our approach follows that of \cite{Laz} closely.

In what follows $X$ is a geometrically irreducible projective variety over a number field $\KK$.  
We let $L$ denote a line bundle on $X$ and defined over $\KK$.  
We denote by $X_{\overline{\KK}}$ and $L_{\overline{\KK}}$ their base change to $\overline{\KK}$.  Note that the context that we consider here is slightly more general than that of Setting \ref{Set:up}.
  
Recall the \emph{semigroup} of $L$
$$
\mathrm{N}(X,L) := \{ m \geq 0 : \H^0(X,L^{\otimes m}) \not = 0 \} \text{.}
$$
When
$
\mathrm{N}(X,L) \not = (0) \text{,}
$
all sufficiently large elements of $\mathrm{N}(X,L)$ are multiples of a largest single natural number
$
e = e(L) \geq 1 \text{;}
$
this is $L$'s \emph{exponent}.

Let 
$$\pi \colon Y \rightarrow X_{\overline{\KK}}$$ 
be the normalization of $X_{\overline{\KK}}$.   Then the \emph{Iitaka dimension} of $L$ can be described as
$$
\kappa(X,L) = \kappa(Y,\pi^*L) := \max_{m \in \mathrm{N}(Y,\pi^*L)} \{\dim \phi_{|\pi^*L^{\otimes m}|} (Y) \} \text{.}
$$
Here
\begin{equation}\label{rational:map:eqn}
\phi_m = \phi_{|\pi^*L^{\otimes m}|} \colon Y \dashrightarrow \PP^{n_m}
\end{equation}
is the rational map that is defined by the complete linear series $|\pi^*L^{\otimes m}_{\overline{\KK}}|$.  In what follows, let $Y_{m}$ be the closure of the image of the rational map \eqref{rational:map:eqn}.  

Recall that $L$ is said to be \emph{big} when 
$$
\kappa(X,L) = \dim X \text{.}
$$

When $\kappa(X,L) \geq 0$, given $m \geq 1$, let $\operatorname{Bs}(|mL_{\overline{\KK}}|)$ be the base locus of $|L^{\otimes m}_{\overline{\KK}}|$.  (If $|L^{\otimes m}| = \emptyset$, then $\operatorname{Bs}(|mL_{\overline{\KK}}|) = X_{\overline{\KK}}$.)  The \emph{stable base locus} of $L$ is the Zariski closed subset
$$
\operatorname{Bs}(L) = \bigcap_{m \geq 1} \operatorname{Bs}(|mL_{\overline{\KK}}|) \text{;}
$$
there exists a positive integer $m_0 > 0$ which has the property that 
$$
\operatorname{Bs}(L) = \operatorname{Bs}(|mm_0L_{\overline{\KK}}|) 
$$
for all $m \gg 0$, \cite[Proposition  2.1.21]{Laz}.

For later use, we recall the main theorem about \emph{Iitaka fibrations}.  (See, for instance \cite[Theorem 2.1.33]{Laz} or \cite[Lemma 1.2]{Mori:1985}.)  In particular, if $\kappa(X,L) > 0$, then for all sufficiently large $m \in \mathrm{N}(X,L)$, the rational mappings
$$
\phi_m = \phi_{|\pi^*L^{\otimes m}|} \colon Y \dashrightarrow Y_m \subseteq \PP^{n_m}
$$
are all birationally equivalent to some \emph{algebraic fibre space}
$$
\phi_{\infty} \colon X_{\infty} \rightarrow Y_{\infty} 
$$
between normal projective varieties $X_{\infty}$ and $Y_{\infty}$.  Especially, the morphism $\phi_{\infty}$ is surjective and has connected fibres.  It is unique up to birational equivalence and is called the \emph{Iitaka fibration} of $L$.

\section{A formulation of the Parametric Subspace Theorem}\label{Parametric:Subspace}
Fix a finite subset $S \subset M_{\KK}$ 
and for each place $v \in S$ fix a collection of $\FF$-linearly independent linear forms
\begin{equation}\label{twist:height:linear:forms}
\ell_{v0}(x), \dots, \ell_{vn}(x) \in \FF[x_0,\dots,x_n] \text{.}
\end{equation}

Fix a real number 
$Q \geq 1$ 
together with a collection of real numbers $c_{vi} \in \RR\text{,}$
for all $v \in S$ 
and $i = 0,\dots,n$, which have the property that
$$
\sum_{i=0}^n c_{vi} = 0 \text{.}
$$

For points of projective $n$-space $\mathbf{x} \in \PP^n(\KK)$ 
put
\begin{equation}\label{twist:height}
H_Q(\mathbf{x}) := \left( \prod_{v \in S} \left(  \max_{0 \leq i \leq n} ||\ell_{vi}(\mathbf{x})||_{v,\KK}  \cdot Q^{-c_{vi}} \right) \right) \cdot H_{\Osh_{\PP^n}(1)}(\mathbf{x}) \text{.}
\end{equation}
This is the \emph{twisted multiplicative height} of 
$\mathbf{x} \in \PP^n(\KK)\text{,}$
with respect to the linear forms \eqref{twist:height:linear:forms} and the real numbers $Q$, $c_{vi}$, for $v \in S$ and $i =0,\dots,n$.  Here $||\ell_{vi}(\mathbf{x})||_{v,\KK}$ is defined as in \eqref{Weil:function:eqn4}.

Note that the twisted multiplicative height \eqref{twist:height} may also be described as
\begin{equation}\label{twist:height:projective}
H_Q(\mathbf{x}) = \prod_{v \in S} \left( \max_{0 \leq i \leq n} |\ell_{vi}(\mathbf{x})|_{v,\KK} \cdot Q^{-c_{vi}} \right) \cdot \prod_{v \not \in S} |\mathbf{x}|_v \text{.}
\end{equation} 
(Compare with \cite[Equation (1.3), p. 514]{Evertse:Ferretti:2013}.)

That the twisted height function 
\eqref{twist:height:projective} can be expressed in the form \eqref{twist:height} is a key point to what follows.

In the proof of our main results, our starting point is the following expanded form of the Parametric Subspace Theorem from \cite[p. 515]{Evertse:Ferretti:2013}.

\begin{theorem}[Parametric Subspace Theorem {\cite{Evertse:Ferretti:2013}}]\label{Parametric:Subspace:Thm}  With the notation and hypothesis as above let $\epsilon > 0$.  Then there exists a real number $Q_0 > 1$ and a finite collection of proper linear subspaces
$$
T_1,\dots,T_{t} \subsetneq \PP^n_{\KK}
$$
which are such that for all $Q \geq Q_0$ there is a subspace
$$
T_{j_Q} \in \{T_1,\dots,T_{t}\}
$$
which contains all 
$
\mathbf{x}  \in \PP^n(\KK)
$
which verify the twisted height inequality
\begin{equation}\label{parametric:twisted:height:inequality}
H_Q(\mathbf{x}) \leq Q^{-\epsilon} \text{.}
\end{equation}
Here, $H_Q(\mathbf{x})$ is the twisted height of $\mathbf{x}$ as defined in \eqref{twist:height}.
\end{theorem}
\begin{proof}
The case that each of the linear forms $\ell_{vi}(x)$ has coefficients in $\KK$ follows as a special case of the Absolute Parametric Subspace Theorem \cite[p. 515]{Evertse:Ferretti:2013}.  To treat the more general case, where the $\ell_{vi}(x)$ have coefficients in $\FF$, we argue as in \cite[Remark 7.2.3]{Bombieri:Gubler}.  (See also the arguments given in  \cite[Proposition 2.1]{Grieve:2018:autissier} and \cite[Proof of Theorem 5.2]{Grieve:Function:Fields}.)

First, there is no loss in generality by assuming that $\FF / \KK$ is a Galois extension.  Let $S'\subset M_{\FF}$ be defined by the condition that 
$$
S' := \{w \in M_{\FF} : w \mid v \text{ and } v \in S \} \text{.}
$$
For each $v \in S$, fix $v' \in S'$ with $v' \mid v$ and consider the absolute value
$$
|\cdot|_{v',\KK} := |\N_{\FF_{v'}/\KK_v}(\cdot)|_v^{\frac{1}{[\FF_{v'}:\KK_v]}} \text{.}
$$
Now, if $w \in S'$ and $w \mid v$, then there exists $\sigma \in \operatorname{Gal}(\FF/\KK)$ with $w = v' \circ \sigma^{-1}$; set
$$
s'_{wi} := \sigma(s_{vi}) \text{ 
for $i = 0,\dots,n$.  }
$$

Fix $\epsilon > 0$ and consider for real numbers $Q \geq 1$ the twisted height function
$$
H_Q'(\mathbf{x}) = \prod_{w \in S'} \left( \max_{0 \leq i \leq n} |\ell_{vi}(\mathbf{x})|_{w} \cdot Q^{-c_{wi}} \right) \cdot \prod_{w \not \in S'} |\mathbf{x}|_w 
$$
defined for points
$
\mathbf{x}  \in \PP^n(\FF) \text{.}
$

Working over $\FF$ the conclusion of the Parametric Subspace Theorem, \cite[p. 515]{Evertse:Ferretti:2013}, is that there exists a real number  $Q_0 > 1$ and a finite collection of proper linear subspaces 
$$
T'_1,\dots,T_t' \subsetneq \PP^n_{\FF}
$$
such that for all $Q \geq Q_0$ there is a subspace
$$
T'_{j_Q} \in \{T_1',\dots,T_t'\}
$$
which contains those
$
\mathbf{x} \in \PP^n(\FF)
$
which have the property that
$$
H_Q'(\mathbf{x}) \leq Q^{-\epsilon} \text{.}
$$

On the other hand, if
$
\mathbf{x} \in \PP^n(\KK)
$
then
$$
H_Q(\mathbf{x}) = H_Q'(\mathbf{x}) \text{.}
$$
The conclusion desired by Theorem \ref{Parametric:Subspace:Thm} then follows by replacing each of the linear subspaces 
$$
T'_i  \subsetneq \PP^n_{\FF}
$$
for $i=1,\dots,t$, by 
$$
T_i := \operatorname{span}_{\KK} \left\{ \mathbf{x} \in \PP^n(\KK) : \mathbf{x} \in T'_i \right\} \text{.}
$$
\end{proof}

\section{Linear systems, rational maps and $v$-adic distances}\label{linear:systems:rational:maps:v:adic:distances}
In this section we give a construction and define auxiliary concepts which we require to formulate our logarithmic  Parametric Subspace Theorem for big linear systems (Theorem \ref{logarithmic:parametric:subspace:thm}).  They build on several viewpoints including \cite[Proposition 4.2]{Autissier:2011}, \cite[Theorem 2.10]{Ru:Vojta:2016}, \cite[Proposition 2.1]{Grieve:2018:autissier}, or \cite[Theorem 3.3]{Grieve:points:bounded:degree}.

Working over a base number field $\KK$ let 
$$
0 \not = V \subseteq \H^0(X,L)
$$
be a nonzero subspace for $L$ an effective line bundle on a geometrically irreducible projective variety $X$.  Let $\operatorname{Bs}(|V|)$ be the base locus of the linear system $|V|$ and $\Ish$ its ideal sheaf.  Fix a basis $s_0,\dots,s_n$ for $V$.  

Let 
$$
\phi \colon X \dashrightarrow \PP^n_{\KK}
$$
be the rational map that is determined by the sections $s_0,\dots,s_n$ and
$$
\phi' \colon X' \rightarrow \PP^n_{\KK}
$$
its extension for 
$$
\pi \colon X' := \operatorname{Bl}_{\Ish}(X) \rightarrow X
$$
the blowing-up of $X$ along $\Ish$.  

Note that the sections $s_0,\dots,s_n$ generate $L$ over the Zariski open subset
$$
U := X \setminus \operatorname{Bs}(|V|) \text{.}
$$
If 
$x \in X\text{,}$ 
then
$$
\Ish_x \simeq \Osh_{X,x}
$$
if and only if $x \in U\text{.}$  
There is the following commutative diagram 

\begin{equation}\label{resolve:base:locus}
\begin{tikzcd}
X' \arrow[drr, "\phi ' "] \arrow[d,"\pi"'] & & \\
X \arrow[r, hookleftarrow] \arrow[rr, dashrightarrow, bend right =20, "\phi"'] & U  \arrow[r]& \PP^n_{\KK} \text{.}
\end{tikzcd}
\end{equation}

The following concept is from \cite{Grieve:points:bounded:degree}. 

\begin{defn}[{\cite[Definition 3.1]{Grieve:points:bounded:degree}}]\label{linear:section:defn} The \emph{proper linear sections}
$$
\Lambda \subsetneq X
$$
of $X$ with respect to $|V|$
are described by 
$$
\Lambda = \pi \left( \phi ^{' -1} (T) \right)
$$
for proper linear subspaces
$$
T \subsetneq \PP^n_{\KK} \text{.}
$$
\end{defn}

For later use, we also define the concept of \emph{density for rational points with respect to a given linear system}.  It builds on earlier ideas of Schmidt \cite[p. 706]{Schmidt:1993} and Evertse \cite[p. 240]{Evertse:1996}.

\begin{defn}\label{density:rational:points} Consider  a non-empty subset of $X(\KK)$.
It is called \emph{dense} with respect to the linear system $|V|$ if it is contained in no finite union of $|V|$'s proper linear sections.
\end{defn}

Returning to the topic of resolving the locus of indeterminacy of the linear system $|V|$,  given a finite extension field $\FF / \KK$ consider base  change of the commutative diagram \eqref{resolve:base:locus}
\begin{equation}\label{resolve:base:locus:base:change}
\begin{tikzcd}
X'_{\FF} \arrow[drr, "\phi_{\FF} ' "] \arrow[d,"\pi_{\FF}"'] & & \\
X_{\FF} \arrow[r, hookleftarrow] \arrow[rr, dashrightarrow, bend right =20, "\phi_{\FF}"'] & U_{\FF}  \arrow[r]& \PP^n_{\FF} \text{.}
\end{tikzcd}
\end{equation}

Recall that the global sections
$$
\pi^* s_i \in \H^0(X', \pi^*L)
$$
generate a line bundle $L'$ on $X'$.  It is a coherent subsheaf of $\pi^* L$.  In what follows, let $s'_i$ be the global section of $L'$ that is determined by $\pi^* s_i$.

Now the line bundle $L'$ and the global sections
$$
 s_i' \in \H^0(X',L')
$$
define the morphism $\phi'$, in \eqref{resolve:base:locus}, and determine the morphism $\phi'_{\FF}$ in the diagram \eqref{resolve:base:locus:base:change}.  The restriction of $\phi'$ to $\pi^{-1}(U)$ corresponds to $\phi$ via the natural isomorphism
$$
\pi \colon \pi^{-1}(U) \xrightarrow{\sim} U \text{.}
$$

As a consequence for local Weil functions, if 
$$
\ell(x) = a_0 x_0 + \hdots + a_n x_n\in \H^0(\PP^n_{\FF}, \Osh_{\PP^n_\FF}(1)) \text{, }
$$
for $a_i \in \FF$ and $i = 0,\dots,n$, is a linear form that pulls back to
$$
s' = a_0s_0'+\hdots+a_ns_n' \in \H^0(X'_{\FF},L'_{\FF})
$$
and corresponds to 
$$
s = a_0s_0+\hdots+a_ns_n \in V_{\FF} 
$$ 
then, for each place $v \in S$
\begin{equation}\label{Weil:function:eqn:resolve:base:locus}
\lambda_{s}(x,v) = \lambda_{\pi^*_{\FF} s}(x',v) = \lambda_{s'}(x',v) = \lambda_{\ell(x)}(\phi'_{\FF}(x'),v) 
\end{equation}
for all
\begin{equation}\label{domain:1}
x \in \left( X \setminus \left( \operatorname{Bs}(|V|) \bigcup \operatorname{Supp}(\operatorname{div}(s)) \right)\right)(\KK) 
\end{equation}
where 
\begin{equation}\label{domain:2}
x' = \pi^{-1}(x) \in \left( X' \setminus \pi^{-1}\left( \operatorname{Bs}(|V|) \bigcup \operatorname{Supp}(\operatorname{div}(s)) \right) \right)(\KK) \text{.}
\end{equation}

Finally, in terms of height functions, for all such $x$ and $x'$, as above in \eqref{domain:1} and \eqref{domain:2}, respectively, it holds true that
\begin{equation}\label{height:function:eqn:resolve:base:locus}
h_L(x) = h_{L'}(x') + \mathrm{O}(1) = h_{\Osh_{\PP^n}(1)}(\phi'(x')) + \mathrm{O}(1) \text{.}
\end{equation}

\section{Proof of Theorem \ref{logarithmic:parametric:subspace:thm}}\label{parametric:proof}

The first step in the proof of Theorem \ref{logarithmic:parametric:subspace:thm} is to consider the logarithmic form of Theorem \ref{Parametric:Subspace:Thm}.  We express its conclusion in terms of local Weil and  logarithmic height functions.

\begin{proposition}\label{log:Parametric:Subspace:Thm}
Fix a finite subset $S \subset M_{\KK}$.  For each $v \in S$ and all $i = 0,\dots,n$, fix a collection of linearly independent linear forms
$$
\ell_{vi}(x) \in \FF[x_0,\dots,x_n] 
$$
together with a collection of real numbers $c_{vi} \in \RR$, for $v \in S$ and $i = 0,\dots,n$, which have the property that
$$
\sum_{i=0}^n c_{vi} = 0 
$$
for all $v \in S$.  For each of the linear forms $\ell_{vi}(x)$, let $\lambda_{\ell_{vi}(x)}(\cdot,v)$ be the local Weil function with respect to $v$ given by \eqref{Weil:function:eqn10}.

Let $\delta > 0$.  Then there exist a real number $Q_0 > 1$ and a finite collection of proper linear subspaces
$$
T_1,\dots,T_{t} \subsetneq \PP^n_{\KK}
$$
such that for all $Q \geq Q_0$ there is a subspace
$$
T_{j_Q} \in \{T_1,\dots,T_{t}\}
$$
which contains all 
$$
\mathbf{x}  \in \left(\PP^n(\KK) \setminus \left( \bigcup_{\substack{v \in S \\ i = 0,\dots,n} } \operatorname{Supp}(\ell_{vi}) \right)\right)(\KK)
$$
which verify the logarithmic twisted height inequality
$$
\sum_{v \in S} \min_{0 \leq i \leq n} \left( \lambda_{\ell_{vi}(x)}(\mathbf{x},v)  + c_{vi} \cdot \log(Q) \right)  
\geq h_{\Osh_{\PP^n}(1)}(\mathbf{x}) + \delta \log(Q) \text{.}
$$
\end{proposition}
\begin{proof}
Apply $-\log(\cdot)$ to the multiplicative twisted height inequality \eqref{parametric:twisted:height:inequality}, which is given in Theorem \ref{Parametric:Subspace:Thm}.  The result is that
\begin{equation}\label{twisted:height:key:eqn}
-\log\left(H_Q(\mathbf{x})\right) \geq \delta \cdot \log(Q) \geq 0  \text{.}
\end{equation}
The conclusion desired by Proposition \ref{log:Parametric:Subspace:Thm} then follows from the conclusion of Theorem \ref{Parametric:Subspace:Thm} since
$$
- \log\left( H_Q(\mathbf{x}) \right) = - \log\left( \prod_{v \in S}
 \left( 
\max_{0 \leq i \leq n} \frac{| \ell_{vi}(\mathbf{x})|_{v,\KK} }{|\mathbf{x}|_{v,\KK}} \cdot Q^{-c_{vi}} 
\right) \right) - h_{\Osh_{\PP^n}(1)}(\mathbf{x}) \text{.}
$$
Indeed, the quantity
$$
- \log \left( \prod_{v \in S} \left( \max_{0 \leq i \leq n} \frac{|\ell_{vi}(\mathbf{x})|_{v,\KK} }{|\mathbf{x}|_{v,\KK} } \cdot Q^{-c_{vi}}\right)\right)
$$
can be rewritten as
$$
\sum_{v \in S} \min_{0 \leq i \leq n} \left( - \log \left( 
\frac{ |\ell_{vi}(\mathbf{x})|_{v,\KK}}{|\mathbf{x}|_{v,\KK} }
 \right) + c_{vi} \cdot \log(Q) \right) \text{.}
 $$
In light of this, the twisted height inequality \eqref{twisted:height:key:eqn} can thus be rewritten in the form
$$
\sum_{v \in S} \min_{0 \leq i \leq n} \left( \lambda_{\ell_{vi}(x)}(\mathbf{x},v)  + c_{vi} \cdot \log(Q) \right)  
\geq h_{\Osh_{\PP^n}(1)}(\mathbf{x}) + \delta \cdot \log(Q) \text{.}
$$
\end{proof}

We now use Proposition \ref{log:Parametric:Subspace:Thm} to prove Theorem \ref{logarithmic:parametric:subspace:thm}. 
\begin{proof}[Proof of Theorem \ref{logarithmic:parametric:subspace:thm}]
Fix a basis $s_0,\dots,s_n$ for $V$.  Let $\operatorname{Bs}(|V|)$ be the base locus of the linear system $|V|$.  We now apply the considerations of  Section \ref{linear:systems:rational:maps:v:adic:distances} within our current context, especially the relations \eqref{Weil:function:eqn:resolve:base:locus} and \eqref{height:function:eqn:resolve:base:locus}.   

Fix linear forms
$$
\ell_{vi}(x) \in \H^0(\PP^n_{\FF},\Osh_{\PP^n_{\FF}}(1))
$$
for all $v \in S$ and all $i = 0,\dots,n$, which pull back to $s_{vi}'$ under $\phi'_{\FF}$ and correspond to $\pi^*_{\FF}s_{vi}$.  Then 
we may also write
 for each place $v \in S$ and each $i=0,\dots,n$
\begin{equation}\label{Weil:function:eqn:resolve:base:locus:1}
\lambda_{s_{vi}}(x,v) = \lambda_{\pi^*_{\FF} s_{vi}}(x',v) = \lambda_{s_{vi}'}(x',v) = \lambda_{\phi'^*_{\FF} \ell_{vi}(x)}(x',v) 
\end{equation}
where
$$
x \in \left( X \setminus \left( \operatorname{Bs}(|V|) \bigcup \operatorname{Supp}(\operatorname{div}(s_{vi})) \right)\right)(\KK) 
$$
and 
$$
x' = \pi^{-1}(x) \in \left( X' \setminus \pi^{-1}\left( \operatorname{Bs}(|V|) \bigcup \operatorname{Supp}(\operatorname{div}(s_{vi})) \right) \right)(\KK) \text{.}
$$

Theorem \ref{logarithmic:parametric:subspace:thm} then follows from the relation \eqref{Weil:function:eqn:resolve:base:locus:1} together with Proposition \ref{log:Parametric:Subspace:Thm} applied to the linear forms 
$$\ell_{v0}(x),\dots,\ell_{vn}(x) \in \FF[x_0,\dots,x_n] \text{ for all $v \in S \text{.}$}
$$
\end{proof}

\section{Proof of Theorem \ref{logarithmic:FW:subspace:thm}}

Similar to the approach of \cite[p.~ 514]{Evertse:Ferretti:2013}, the conclusion of Theorem  \ref{logarithmic:FW:subspace:thm} is implied by that of Theorem \ref{logarithmic:parametric:subspace:thm}.

\begin{proof}[Proof of Theorem \ref{logarithmic:FW:subspace:thm}]
Put
$$
\epsilon := - 1 + \frac{1}{n+1}\left( \sum_{v \in S} \sum_{i=0}^n d_{vi} \right) \text{.}
$$
For each $v \in S$ 
and all $i = 0,\dots,n$ set
$$
c_{vi} := -d_{vi} + \frac{1}{n+1} \sum_{j=0}^n d_{vj} \text{.}
$$
Then
$\epsilon > 0$
and
$$
\sum_{i=0}^n c_{vi} = 0 \text{ 
for all $v \in S$.
}
$$

Let 
\begin{equation}\label{solution:domain:eqn1}
x \in \left( X \setminus \left( \operatorname{Bs}(|V|) \bigcup \bigcup_{\substack{v \in S \\ i = 0,\dots,n} }\operatorname{Supp}(s_{vi}) \right) \right)(\KK)
\end{equation}
be a solution of the system of inequalities
$$
\lambda_{s_{vi}}(x,v) - d_{vi} \cdot h_L(x) + \mathrm{O}_v(1) \geq 0 
$$
for all $i = 0,\dots,n$ and all $v \in S$, and put 
$$
Q = \exp(h_L(x)) \text{.}
$$
Then the inequalities
$$
\sum_{v \in S} \left( \lambda_{s_{vi}}(x,v) + c_{vi} \cdot \log(Q) \right) \geq h_L(x) + \epsilon \cdot \operatorname{log}(Q) + \mathrm{O}(1)  
$$
for $i = 0,\dots,n$
are valid.  Thus, by considering all such solutions $x$, as above in \eqref{solution:domain:eqn1}, with sufficiently large height
$$h_L(x) \geq \log(Q_0) \gg 0$$ 
for a suitable real number $Q_0 > 1$,
the conclusion of Theorem \ref{logarithmic:FW:subspace:thm} follows from that of Theorem \ref{logarithmic:parametric:subspace:thm}.  Here, we employ the Northcott theorem, for big line bundles, compare with \cite[Theorem 2.4.9]{Bombieri:Gubler}, in order to conclude that there will be only finitely many solutions $x$ (and hence finitely many exceptional subspaces to add) which have sufficiently small height.
\end{proof}

\section{Proof of Theorem \ref{logarithmic:linear:scattering:subspace:thm} }\label{Log:Subspace:Proof}

By modifying the arguments of \cite[Section 21]{Evertse:Schlickewei:2002}, as suggested in \cite[p.~514]{Evertse:Ferretti:2013}, Theorem \ref{logarithmic:linear:scattering:subspace:thm} follows from Theorem \ref{logarithmic:FW:subspace:thm}, in light of Lemma \ref{Evertse:scattering:lemma4} below.  Lemma \ref{Evertse:scattering:lemma4} is a special case of the logarithmic form of \cite[Lemma 4]{Evertse:1984}.  

The proof of Theorem \ref{logarithmic:linear:scattering:subspace:thm} is interesting as it  involves a partitioning of the solution set into subsets.   Each subset solves a certain simultaneous system of Diophantine arithmetic inequalities.  Since only finitely many such subsets are required, Theorem \ref{logarithmic:linear:scattering:subspace:thm} thus follows upon repeated application of Theorem \ref{logarithmic:FW:subspace:thm}.  Such conceptual reasoning will be made more explicit throughout the proof of Theorem \ref{logarithmic:linear:scattering:subspace:thm}.  We prove Corollary \ref{general:position:logarithmic:linear:scattering:subspace:thm} after first proving Theorem \ref{logarithmic:linear:scattering:subspace:thm}.

\begin{lemma}[Compare with {\cite[Lemma 21.1]{Evertse:Schlickewei:2002}}  or {\cite[Lemma 4]{Evertse:1984}}]\label{Evertse:scattering:lemma4}
Fix a sufficiently small positive real number $c$, 
$0 < c \ll 1\text{,}$ 
and let $I$ be a finite set.
Then the set
$$
\mathcal{R}(c) := \left\{ \mathbf{c} = (c_i)_{i \in I}: c_i \in \RR_{\geq 0} \text{ and } \sum_{i \in I} c_i = c \right\} 
$$
admits a finite subset 
$
\mathcal{S}(c) \subseteq \mathcal{R}(c)
$
which has the property that for all 
$
\mathbf{b} = (b_i)_{i \in I} \text{ with $b_i \in \RR_{\geq 0}$}
$
there exists
$
\mathbf{a} = (a_i)_{i \in I}  \in \mathcal{S}(c)
$
which has the property that
$$
b_j \geq a_j \left(  \sum_{i \in I} b_i \right) \text{ 
for all $j \in I$. }
$$
\end{lemma}

\begin{proof}
This is a special case of the logarithmic formulation of \cite[Lemma 4]{Evertse:1984}.
\end{proof}

We now use Lemma \ref{Evertse:scattering:lemma4} to establish Theorem \ref{logarithmic:linear:scattering:subspace:thm}.

\begin{proof}[Proof of Theorem \ref{logarithmic:linear:scattering:subspace:thm}]

Fix a sufficiently small positive real number $\epsilon > 0$.  Our aim is to ascertain qualitative features, that are expressed in terms of $|L|$'s linear sections, of the collection of those
 solutions 
\begin{equation}\label{scattering:solutions:domain}
x \in \left( X \setminus \left( \operatorname{Bs}(|V|) \bigcup \bigcup_{\substack{v \in S \\ i = 0,\dots,n}} \operatorname{Supp}(s_{vi}) \right)\right)(\KK) 
\end{equation}
to the inequality
\begin{align}\label{scattering:system}
\begin{split}
 \sum_{v \in S} \sum_{i=0}^n \lambda_{s_{vi}}(x,v) & \geq (n+1+\epsilon)h_{L}(x) + \mathrm{O}(1)  \end{split}
\end{align}
which have sufficiently large height
$h_L(x) \gg 0$.

In our study of the system \eqref{scattering:system}, similar to the approach of \cite[Section 21]{Evertse:Schlickewei:2002}, we distinguish amongst two classes of solutions.

\begin{itemize}
\item{
{\bf Type I}: are those solutions which admit an index $i$, $0 \leq i \leq n$, which is such that 
$$
 \sum_{v \in S} \lambda_{s_{vi}}(x,v) \geq (n+1+ \epsilon)h_{L}(x) + \mathrm{O}(1) \text{;}
$$
}
\item{
{\bf Type II}: are those solutions which satisfy the inequalities
$$
 \sum_{v \in S} \lambda_{s_{vi}}(x,v) < (n+1+\epsilon) h_{L}(x) + \mathrm{O}(1) \text{ for all $i = 0,\dots,n$.}
$$
}
\end{itemize}

In our study of the {\bf Type I} and {\bf Type II} solutions, by adjusting the constant $\mathrm{O}(1)$, if necessary, there is no loss in generality by assuming that
\begin{equation}\label{non:negative:lin:forms:weil}
\lambda_{s_{vi}}(x,v) \geq 0
\end{equation}
for all $i=0,\dots,n$, all $v \in S$ and all solutions $x$ of the form  \eqref{scattering:solutions:domain}.

\subsection*{Simultaneous Type I inequalities.}
Fix a {\bf Type I} solution $x$.   Then, there exists an index $i$, with $0 \leq i \leq n$ and 
$$
\sum_{v \in S} \lambda_{s_{vi}}(x,v) \geq (n+1+\epsilon)h_{L}(x) + \mathrm{O}(1) \text{.}
$$
Fix such an index $i$ and define the tuple 
\begin{equation}\label{type:one:soltn:tuple}
\mathbf{b} = (b_v)_{v \in S}
\end{equation}
 by the condition that
$$
\lambda_{s_{vi}}(x,v) = b_v h_{L}(x) + \mathrm{O}(1) \text{.}
$$
The tuple \eqref{type:one:soltn:tuple}, which depends on $i$ and $x$, has the property that
$$
b_v \geq 0 \text{ for each $v \in S$}
$$
and
$$
\sum_{v \in S} b_v \geq n+1+ \epsilon \text{.}
$$

We now consider consequences of Lemma \ref{Evertse:scattering:lemma4} applied to the sufficiently small positive real number
$$
c = 1 - \frac{\epsilon}{4(n+1)}
$$
and the finite set $S$.  

Indeed, we deduce from Lemma \ref{Evertse:scattering:lemma4} that the collection of {\bf Type I} solutions can be partitioned into finitely many subsets $\mathcal{S}_{\mathbf{a}}$ where each subset corresponds to a fixed tuple 
\begin{equation}\label{type:one:soltn:tuple:eqn:2}
\mathbf{a} = (a_v)_{v \in S}
\end{equation}
of nonnegative real numbers with the property that
\begin{equation}\label{type:one:soltn:tuple:eqn:2:prime}
\sum_{v \in S} a_v = 1 - \frac{\epsilon}{4(n+1)} \text{.}
\end{equation}

A given subset $\mathcal{S}_{\mathbf{a}}$ with corresponding weight vector \eqref{type:one:soltn:tuple:eqn:2}
consists of those {\bf Type I} solutions $x$ for which their vector \eqref{type:one:soltn:tuple} satisfies the condition that
$$
b_v \geq a_v \sum_{w \in S} b_w \geq a_v(n+1+\epsilon) \text{.}
$$
As a consequence, it follows that 
$$
\lambda_{s_{vi}}(x,v) > a_v(n+1+\epsilon)h_{L}(x) + \mathrm{O}(1)
$$
for all such {\bf Type I} solutions $x$ in the subset corresponding to the tuple \eqref{type:one:soltn:tuple:eqn:2}.

Now we need to do some rewriting of the sections $s_{vi}$ working on the blow-up of the base locus $\operatorname{Bs}(|V|)$, as in Section \ref{linear:systems:rational:maps:v:adic:distances}.  We view these sections $s_{vi}$ as having the property that $\pi^*s_{vi}$, their pullbacks to $X'$, correspond with the sections $s_{vi}'$ which are the pullbacks, with respect to $\phi'_{\FF}$, of the linear forms 
$$
\ell_{vi}(x) = a_{vi0}x_0+\hdots+ a_{vin}x_n \in \FF[x_0,\dots,x_n]
\text{ 
for all $v \in S$ and $i = 0,\dots,n$.
}
$$

Given a pair $(i,v)$, where $v \in S$ and $0 \leq i \leq n$, pick 
$j = j(i,v)$ 
with 
$0 \leq j(i,v) \leq n$ 
and 
$$
\left| a_{vij} \right|_v = \max \{ |a_{vi0}|_v,\dots, |a_{vin}|_v \} \text{.}
$$
In fact, by homogeneity of the linear forms, and by adjusting the constant $\mathrm{O}(1)$, if necessary, there is no loss in generality by assuming that each form $\ell_{vi}(x)$ has the property that
$$
a_{v i j(i,v) } = 1 \text{.}
$$

Consider, for each $v \in S$, the collection of linear forms $\ell_{vi}(x)$, $x_k$ for $k = 0,\dots,n$ and $k \not = j(i,v)$.  Relabel this collection of linear forms as
$$
m_{v0}(x) = \ell_{vi}(x),m_{v1}(x),\dots,m_{vn}(x) \text{.}
$$

Then for each tuple \eqref{type:one:soltn:tuple:eqn:2} define the tuple of real numbers
\begin{equation}\label{type:one:soltn:tuple:eqn:3}
 (e_{vi})_{\substack{v \in S \\ i = 0,\dots,n }}
\end{equation}
by the condition that
\begin{equation}\label{type:one:soltn:tuple:eqn:4}
e_{vi} = \begin{cases}
 a_v(n+1+\epsilon) & \text{ for $i = 0$; and } \\
0 & \text{ for $i = 1,\dots,n$.}
\end{cases}
\end{equation}
Upon adjusting the constant $\mathrm{O}(1)$, if necessary,  the tuple \eqref{type:one:soltn:tuple:eqn:3} and the linear forms 
$$m_{v0}(x),\dots,m_{vn}(x)$$ 
are such that all {\bf Type I} solutions in the subset $\mathcal{S}_{\mathbf{a}}$ 
satisfy the inequality 
\begin{equation}\label{type:one:soltn:tuple:eqn:5}
- \log \left( \frac{|m_{vi}(x)|_{v,\KK} }{|x|_{v,\KK} } \right) \geq  e_{vi} h_{\Osh_{\PP^n}(1)}(x) + \mathrm{O}(1)
\end{equation}
for each pair $(i,v)$, where $v \in S$ and $i = 0,\dots,n$.

Now, observe that, by \eqref{type:one:soltn:tuple:eqn:2:prime}, the tuple \eqref{type:one:soltn:tuple:eqn:3} 
satisfies the condition that
$$
\sum_{v \in S} \sum_{i=0}^n e_{vi} 
 =  (n+1+\epsilon)\left( \sum_{v \in S} a_v \right) 
>  n + 1 \text{.}
$$

Thus, to summarize, we have shown that the collection of {\bf Type I} solutions $x$ in \eqref{scattering:solutions:domain} admits a decomposition into finitely many subsets such that if \eqref{type:one:soltn:tuple:eqn:3}
is the tuple corresponding to a given subset $\mathcal{S}_{\mathbf{a}}$, defined by \eqref{type:one:soltn:tuple:eqn:4}, then all solutions to this subset are solutions of the simultaneous system \eqref{type:one:soltn:tuple:eqn:5}.
The {\bf Type I} solutions are thus described by finitely many applications of Theorem \ref{logarithmic:FW:subspace:thm}. 

\subsection*{Simultaneous Type II inequalities.}  Fix a {\bf Type II} solution $x$ in \eqref{scattering:solutions:domain}.  Recall that these are the solutions for which 
\begin{equation}\label{typeII:eqns}
 \sum_{v \in S} \lambda_{s_{vi}}(x,v) < (n+1+\epsilon)h_{L}(x) + \mathrm{O}(1) \text{ for all $i = 0, \dots, n$.}
\end{equation}

Again, we need to do some rewriting of the sections $s_{vi}$ working on the blow-up of the base locus $\operatorname{Bs}(|V|)$ as in Section \ref{linear:systems:rational:maps:v:adic:distances}. Thus, as in our study of the {\bf Type I} solutions, we view these sections $s_{vi}$ as having the property that $\pi^*s_{vi}$, their pullbacks to $X'$ correspond to sections $s_{vi}'$ which are the pullbacks, with respect to $\phi'_{\FF}$, of the linear forms 
$$
\ell_{vi}(x) = a_{vi0}x_0+\hdots+ a_{vin}x_n \in \FF[x_0,\dots,x_n]  \text{ for $v \in S$ and $i = 0,\dots,n$.  }
$$

Now, by adjusting the constant $\mathrm{O}(1)$, without loss of generality, we may assume that 
$$
0 \leq \mathrm{O}(1) < n(n+1+\epsilon)h_L(x) 
$$
for all {\bf Type II} solutions $x$ of the form \eqref{scattering:solutions:domain}.
Moreover, we fix a real number 
$A \geq 0\text{,}$ 
which has the property that
$$
\mathrm{O}(1) = n (n+1+\epsilon)A \text{.}
$$
Further, for each place $v \in S\text{,}$ 
we may choose nonnegative real numbers $d_v$ so that
$$
\sum_{v \in S} d_v = 1 \text{.}
$$

Finally, fix a scaled nonnegative additive decomposition of the constant $\mathrm{O}(1)$;  assume that this decomposition is indexed by the set of places $S$.  In more precise terms
$$
\mathrm{O}(1) = \frac{1}{n+1} \sum_{v \in S} \delta_v 
\text{ 
where
$
\delta_v \geq 0 \text{.}
$
}
$$
Then we may assume that the nonnegative real numbers $d_v$ satisfy the relation that
$$
\delta_v = d_v n (n+\epsilon+1) A \text{ for each $v \in S$.}
$$

Note that we may assume that
$$
0 \leq A \ll h_{L}(x) 
$$
for some sufficiently large constant $A$.   (By assumption, we are considering those {\bf Type II} solutions which have sufficiently large height.)

Now, for each place $v \in S$, put
$$
b_v = \frac{1}{n+1} d_v n(n+1+\epsilon) \text{.}
$$
Then
$$
\sum_{v \in S} \sum_{i=0}^n b_v = n (n+1+\epsilon) \text{.}
$$
Moreover
$$
\frac{1}{n+1}\delta_v \geq b_v h_{L}(x)
\text{ 
for all $v \in S$.
}
$$

Combined, the above discussion, together with \eqref{scattering:system}, implies that for all type {\bf Type II} solutions with sufficiently large height
$$
\sum_{v \in S} \sum_{i=0}^n \lambda_{s_{vi}}(x,v) 
\geq 
(n+1)(n+1+\epsilon)h_{L}(x) + \sum_{v \in S} \delta_v - \sum_{v \in S} \sum_{i=0}^n b_v h_{L}(x) \text{.}
$$

Now, consider the implications of Lemma \ref{Evertse:scattering:lemma4} applied to the sufficiently small positive real number
$$ c = 1 - \frac{\epsilon}{4(n+1)^2} $$
and the finite set $S \times \{0,\dots,n\}$. 
The conclusion is that the collection of {\bf Type II} solutions admits a partition into finitely many subsets so that the following is true:

For each such subset $\mathcal{S}_{\mathbf{b}}$, there exists a tuple of nonnegative numbers
\begin{equation}\label{type:2:solt:eqn1}
\mathbf{b} = (b_{vi})_{\substack{v \in S \\
i = 0, \dots ,n }
}
\end{equation}
which has the property that
$$
\sum_{v \in S} \sum_{i=0}^n b_{vi} = 1 - \frac{\epsilon}{4(n+1)^2}
$$
and which is such that if $x$, as in \eqref{scattering:solutions:domain}, is a {\bf Type II} solution that lies in this subset $\mathcal{S}_{\mathbf{b}}$ 
then
$$
\lambda_{s_{vi}}(x,v)  \geq   b_{vi}(n+1)(n+1+\epsilon)h_{L}(x) + \frac{1}{n+1}\delta_v - b_v h_{L}(x) 
$$
for all $v \in S$ and all $i = 0,\dots,n$.

Now, given such a tuple \eqref{type:2:solt:eqn1}, 
define the tuple 
\begin{equation}\label{type:2:solt:eqn2}
(e_{vi})_{\substack{v \in S \\ i = 0,\dots, n} }
\end{equation}
by the condition that
$$
e_{vi} =  b_{vi}((n+1)(n+1+\epsilon)) - b_v \text{.}
$$
Then the {\bf Type II} solutions that lie in the subset $\mathcal{S}_{\mathbf{b}}$ satisfy the inequality that
$$
\lambda_{s_{vi}}(x,v) \geq  e_{vi} h_{L}(x) + \frac{1}{n+1} \delta_v \text{.}
$$
Moreover, note that, by construction, the tuple \eqref{type:2:solt:eqn2} has the property that
\begin{align*}
\begin{split}
\sum_{v \in S} \sum_{i=0}^n e_{vi} 
&= (n+1)(n+1+\epsilon)\left(1 - \frac{\epsilon}{4(n+1)^2} \right) - n(n+1+\epsilon) 
\\
& > n + 1 \text{.}
\end{split}
\end{align*}

Thus, after adjusting the given constant $\mathrm{O}(1)$, if necessary, the above discussion implies that the nature of the collection of {\bf Type II} solutions 
are thus described by finitely many applications of Theorem \ref{logarithmic:FW:subspace:thm}.
\end{proof}

Having established the case that  $n = n_v$, for all $v \in S$, in the form of Theorem \ref{logarithmic:linear:scattering:subspace:thm}, the case that $n_v > n$, then follows by adapting the argument of \cite[Proof of Theorem 7.2.9]{Bombieri:Gubler} to the language of linear systems.  In particular, Corollary \ref{general:position:logarithmic:linear:scattering:subspace:thm} is established by successive application of Theorem \ref{logarithmic:linear:scattering:subspace:thm}.

\begin{proof}[Proof of Corollary \ref{general:position:logarithmic:linear:scattering:subspace:thm}]
We reduce from the case that $n_v > n$ to finitely many applications of the case that $n_v = n$.  
This is achieved
by partitioning the solutions to the inequality
$$
\sum_{v \in S} \sum_{i=0}^{n_v} \lambda_{s_i}(x,v) \geq (n+1+\epsilon)h_L(x) + \mathrm{O}(1)
$$
into finitely many classes such that after reindexing the sections $s_{v0},\dots,s_{vn}$, if needed, and setting 
$m_{vi} = s_{vi}\text{,}$ 
for $0\leq i \leq n$;  then there exists a constant $C$ which is such that 
$$
\sum_{v \in S} \sum_{i=0}^{n} \lambda_{m_{vi}}(x,v) \geq \sum_{v \in S} \sum_{i=0}^{n_v} \lambda_{s_{vi}}(x,v) - \log(C) > (n+1+\epsilon)h_L(x) -\log(C) \text{.}
$$
By the Northcott property for big line bundles, adjusting $\epsilon$ and excluding at most a finite number of solutions, if required, the desired conclusion of Theorem \ref{logarithmic:linear:scattering:subspace:thm}, for each fixed partition described above, follows from the case that  $n=n_v$.
\end{proof}

\section{ 
Proof of Theorem \ref{Arithmetic:General:Thm:linear:scattering}}\label{Proof:Arithmetic:General:Thm:linear:scattering}

\begin{proof}[Proof of Theorem \ref{Arithmetic:General:Thm:linear:scattering}]
The proof of Theorem \ref{Arithmetic:General:Thm:linear:scattering} is based on the Ru-Vojta filtration construction.  We briefly recall the most important points here and refer to the presentation given in \cite{Grieve:points:bounded:degree} for further details.  

Let 
$d := \dim X\text{,}$ set
$$
\Sigma := \left\{ \sigma \in \{1,\dots,q\} : \bigcap_{j \in \sigma} \operatorname{Supp}(D_j) \not = \emptyset \right\} 
$$
and fix $m \gg 0$, such that 
$\operatorname{Bs}(L) = \operatorname{Bs}(|L^{\otimes m}_{\overline{\KK}}|)\text{.}$

For each $i  = 1,\dots, q$, and each $v \in S$, let  $\lambda_{\mathcal{D}_i}(\cdot,v)$ be a local Weil function for $D_i$ with respect to $v$ and a fixed choice of presentation.

Observe now that, as noted in \cite[p. 985]{Ru:Vojta:2016}, each of the quantities $m_S(x,D_i)/h_L(x)$, for $i = 1,\dots,q$, is bounded for all $x \in X(\KK)$ outside of some proper Zariski closed subset.  In more precise terms, similar to the reduction step from \cite[p. 8]{Grieve:points:bounded:degree}, since $L$ is assumed to be big, it follows from the Northcott property, for big line bundles, that there exist constants $A$ and $B$ which are such that for all but perhaps finitely many points
$$
x \in X(\KK) \setminus \left(\operatorname{Bs}(L) \bigcup \bigcup_{i=1}^q \operatorname{Supp}(D_i) \right)  
$$
if $h_L(x) > B$,
then
$$
\sum_{v \in S} \lambda_{\mathcal{D}_i}(x,v) < A h_L(x) \text{ for all $i = 1,\dots,q$.}
$$

Thus, to prove Theorem \ref{Arithmetic:General:Thm:linear:scattering}, following the approach of \cite[p. 8]{Grieve:points:bounded:degree} and adjusting $\epsilon>0$ if necessary, we fix suitable sufficiently small rational numbers $\beta_1,\dots,\beta_q \in \QQ$, which are such that 
if 
$$
\gamma(L,D_i) := \limsup_{m \to \infty} \frac{m h^0(X,L^{\otimes m})}{\sum_{\ell \geq 1} h^0(X_{\FF},L^{\otimes m}_{\FF} \otimes \Osh_{X_{\FF}}(-\ell D_i))} \text{,}
$$
then 
$
\beta_i < \gamma(L,D_i)^{-1} \text{.}
$

Finally, we also fix a positive integer $b > 0$ and a sufficiently small positive number $\epsilon_1>0$ which are such that the inequality
\begin{equation}\label{Arithmetic:General:Thm:eqn1}
\left( 1 + \frac{d}{b} \right) \max_{1 \leq i \leq q} \frac{ \beta_i m h^0(X,L^{\otimes m}) + m \epsilon_1}{\sum_{\ell \geq 1} h^0(X_{\FF},L^{\otimes m}_{\FF} \otimes \Osh_{X_{\FF}}(- \ell D_i)) } < 1 + \epsilon 
\end{equation}
holds true.

Now, for each $\sigma \in \Sigma$, let
$$
\Delta_\sigma := \left\{ \mathbf{a} = (a_i) \in \prod_{i \in \sigma} \beta_i^{-1} \mathbb{N} : \sum_{i \in \sigma} \beta_i a_i = b \right\}
$$
and for each $\mathbf{a} \in \Delta_{\sigma}$ define for all $t \in \RR_{\geq 0}$
$$
\Ish(t) := \sum_{  \substack{ \mathbf{b} \in \NN^{\# \sigma} \\ \sum_{i \in \sigma} a_ib_i \geq t } } \Osh_{X_{\FF}}\left(- \sum_{i \in \sigma} b_i D_i \right) \text{;}
$$
set
$$
\mathcal{F}(\sigma;\mathbf{a})_t = \H^0(X_{\FF},L^{\otimes m}_{\FF} \otimes \Ish(t)) \subseteq \H^0(X_{\FF},L^{\otimes m}_{\FF})\text{.}
$$
In what follows, we let $\mathcal{B}_{\sigma;\mathbf{a}}$ be a basis of $\H^0(X_{\FF},L^{\otimes m}_{\FF})$ which is adapted to the filtration
$\{\mathcal{F}(\sigma; \mathbf{a})_t\}_{t \in \RR_{\geq 0}}$.

Now, the key technical point, which is exposed in \cite[pp.  9--12]{Grieve:points:bounded:degree}, is that, since the divisors $D_1,\dots,D_q$ intersect properly, over $\FF$, compare with \cite[Definition 2.1 (b)]{Ru:Vojta:2016}, the above filtration construction produces a collection of sections
$$
\{s_1,\dots,s_{k_2}\} \subseteq \H^0(X_{\FF},L^{\otimes m}_{\FF})
$$
which have the property that if $v \in S$, then
\begin{multline}\label{Arithmetic:General:Thm:eqn2}
\frac{b}{b+d} \left( \min_{1 \leq i \leq q} \sum_{\ell \geq  0} \frac{h^0(X_{\FF},L^{\otimes m}_{\FF} \otimes \Osh_{X_{\FF}}(-\ell D_i))}{\beta_i}\right) \sum_{i =1}^q \beta_i \lambda_{\mathcal{D}_j}(\cdot,v) \\
\leq \max_{1 \leq i \leq  k_1} \sum_{j  \in J_i} \lambda_{s_j}(\cdot,v) + \mathrm{O}_v(1) \text{.}
\end{multline}
Here
$J_i\subseteq \{1,\dots,k_1\}$ 
are chosen such that 
$\mathcal{B}_i = \{s_j : j \in J_i\}$
where 
$$
\bigcup_{\sigma;\mathbf{a}} \mathcal{B}_{\sigma;\mathbf{a}} = \mathcal{B}_1 \bigcup \dots \bigcup \mathcal{B}_{k_1} = \{s_1,\dots,s_{k_2}\} \text{.}
$$

On the other hand, it follows from the Subspace Theorem, with linear scattering (Theorem \ref{logarithmic:linear:scattering:subspace:thm}) in the form of Corollary \ref{general:position:logarithmic:linear:scattering:subspace:thm}, that
\begin{equation}\label{Arithmetic:General:Thm:eqn3}
\sum_{v \in S} \max_J \sum_{j \in J}  \lambda_{s_j}(x,v) \leq \left(h^0(X,L^{\otimes m}) + \epsilon_1 \right)h_{L^{\otimes m}}(x) + \mathrm{O}(1) 
\end{equation}
for all $x \in X(\KK)$ outside of a Zariski closed subset $Z$.  This subset $Z$ may be taken to be in the form desired by the conclusion of Theorem \ref{Arithmetic:General:Thm:linear:scattering}.  (In \eqref{Arithmetic:General:Thm:eqn3}, the maximum is taken over all $J \subseteq \{1,\dots,k_1\}$ for which the sections $s_j$, $j \in J$, are linearly independent.)

Combining the above three inequalities, \eqref{Arithmetic:General:Thm:eqn1}, \eqref{Arithmetic:General:Thm:eqn2} and \eqref{Arithmetic:General:Thm:eqn3}, and using the fact that 
$
h_{L^{\otimes m}}(x)  = m h_L(x) 
$
it then follows that
\begin{multline}\label{Arithmetic:General:Thm:eqn4}
\sum_{i=1}^q \beta_i m_S(x,D_i) \leq \\
\left(1 + \frac{d}{b} \right) \max_{1 \leq i \leq q} \left( \frac{\beta_i h^0(X,L^{\otimes m}) + \epsilon_1}{ \sum_{\ell \geq 1} h^0(X_{\FF},L^{\otimes m}_{\FF} \otimes \Osh_{X_{\FF}}(-\ell D_i)) } \right) h_{L^{\otimes m}}(x) + \mathrm{O}(1) \text{.}
\end{multline}
This final inequality \eqref{Arithmetic:General:Thm:eqn4}, may be written in the form
$$
\sum_{i=1}^q \beta_i m_S(x,D) \leq (1 + \epsilon) h_L(x) + \mathrm{O}(1) \text{.}
$$

We have proved existence of constants $\beta(L,D_i)$, for $i = 1,\dots,q$, which are such that the conclusion of Theorem  \ref{Arithmetic:General:Thm:linear:scattering} holds true.  To complete the proof, we replace the $\beta(L,D_i)$ by smaller real numbers as required.
\end{proof}

\section{Asymptotic nature of linear sections}\label{Iitaka:fibration:linear:sections}

We now discuss \emph{asymptotic aspects} of the concept of linear section with respect to a linear series.  To this end, let $X$ be a geometrically irreducible and geometrically normal projective variety over a base number field $\KK$.  Let $L$ be a line bundle on $X$ which has the property that $\kappa(X,L) > 0$.

For simplicity we assume that $L$ has exponent $e = e(L)$ equal to $1$.  For sufficiently large integers $m \in \mathrm{N}(X,L)$, 
with the property that
$$
n_m := \dim |L^{\otimes m}| > 1 \text{,}
$$
let
$
X_m \subseteq \PP^{n_m}
$
be the closure of the image of the rational map
$$
\phi_m = \phi_{|L^{\otimes m}|} \colon X \dashrightarrow \PP^{n_m} \text{.}
$$
Moreover, denote by
$$
\pi_m \colon X'_m \rightarrow X
$$
a resolution of indeterminacies of $\phi_{|L^{\otimes m}|}$.

Fix a suitable sufficiently large integer  
$
\ell_0 \in \mathrm{N}(X,L)  
$
which has the property that
$$
\dim X_{\ell_0} = \kappa(X,L)\text{.}
$$
Fix two relatively prime positive integers $p$ and $q$ with the property that
$$
\dim X_p = \dim X_q = \kappa(X,L) \text{.}
$$
Observe that all sufficiently large positive integers $m \gg 0$ can be written in the form
$$
m = b p^{\ell_0} + c q^{\ell_0} \text{.}
$$
Here, $b,c \geq 1$ are suitable positive integers.

It then follows, via the theory of Iitaka fibrations, \cite{Laz}, that for all such sufficiently large integers $m \gg 0$, there exists a commutative diagram

\begin{center}
\begin{tikzcd}
& X'_{m} \times_{\operatorname{Spec \KK}} \operatorname{Spec} \FF \arrow[dr, "\phi '_m "] \arrow[d,"\pi_{m}"] &  \\
X_{\infty} \arrow[ur, "u_m"] \arrow[d, "\phi_{\infty}"'] \arrow[r, "u_{\infty}"]  & X_{\FF}  \arrow[r, dashrightarrow, "\phi_m"]  &  X_{m} \times_{\operatorname{Spec \KK}} \operatorname{Spec} \FF  \arrow[dll, dashrightarrow,  "\mu_m"] \subseteq \PP^{n_m}_{\FF} \text{.} \\
Y_{\infty}
\end{tikzcd}
\end{center}

Here, $\phi_{\infty}$ is an Iitaka fibration for $L$, and defined over some finite extension $\FF$ of the base number field $\KK$, the morphism $\mu_m$ is generically finite whereas the morphisms $u_m$ and $u_{\infty}$ are birational.

Thus, for all sufficiently large positive integers
$$
m = b p^{\ell_0} + c q^{\ell_0} \gg 0 \text{,}
$$
the proper linear sections of $X$, with respect to $L$, \emph{stabilize} in the sense that each linear section of $X$ with respect to $|L^{\otimes m}|$ can be described as
$$
\pi_m(\phi_m'^{-1}(T)) = u_{\infty}(u_m ^{-1}( \phi^{' -1}_m (T))) \text{,}
$$
for some proper linear subspace 
$T \subsetneq \PP^{n_m}_{\KK}$.

\providecommand{\bysame}{\leavevmode\hbox to3em{\hrulefill}\thinspace}
\providecommand{\MR}{\relax\ifhmode\unskip\space\fi MR }
\providecommand{\MRhref}[2]{%
  \href{http://www.ams.org/mathscinet-getitem?mr=#1}{#2}
}
\providecommand{\href}[2]{#2}

\end{document}